\documentclass{article}

\usepackage{lineno,hyperref}
\usepackage{amsmath}
\usepackage{amssymb}
\usepackage{amsthm}
\modulolinenumbers[5]
\usepackage{enumitem}
\usepackage[all]{xy}
\SelectTips{cm}{12}

\newtheorem{lemma}{Lemma}
\newtheorem*{lemma*}{Lemma}

\newtheorem{theorem}{Theorem}
\newtheorem*{theorem*}{Theorem}

\theoremstyle{definition}

\newcommand\dotminus{\mathbin{\dot{-}}}

\newcommand\dotoplus{\mathbin{\dot{\oplus}}}

\DeclareMathOperator\mat{M}

\DeclareMathOperator\Cent{C}

\DeclareMathOperator\Ker{Ker}

\newcommand\leqt{\trianglelefteq}

\DeclareMathOperator\Spec{Spec}

\newcommand{\up}[2]{{^{#1}\!{#2}}}

\title{
	The Diophantine problem \\
	in isotropic reductive groups
}

\author{
	Egor Voronetsky\thanks{
		This work was performed at the Saint Petersburg Leonhard Euler International Mathematical Institute and supported by the Ministry of Science and Higher Education of the Russian Federation (agreement no. 075--15--2025--343).
	} \\
	Saint Petersburg University,\\
	7/9 Universitetskaya nab.,\\
	St. Petersburg, 199034 Russia
}

\begin{document}
\maketitle

\begin{abstract}
	We begin to study model-theoretic properties of non-split isotropic reductive group schemes. In this paper we show that the base ring
	\( K \) is e-interpretable in the point group
	\( G(K) \) of every sufficiently isotropic reductive group scheme
	\( G \). In particular, the Diophantine problems in
	\( K \) and
	\( G(K) \) are equivalent. We also compute the centralizer of the elementary subgroup of
	\( G(K) \) and the common normalizer of all its root subgroups.
\end{abstract}

\section{Introduction}

Let
\( M \) be a model of some first-order language
\( L \) with equality, i.e.
\( M \) is a set with interpretation of constant symbols, functional symbols, and predicate symbols from
\( L \). For example,
\( M \) may be a group with constant
\( 1 \), operations of multiplication and inversion, and the equality predicate. Recall that an
\textit{elementary formula} with the variables from
\( \vec x = (x_1, \ldots, x_n) \) is a formula in the language
\( L \) consisting of a single predicate symbol with substituted terms involving only variables from
\( \vec x \) and constants from
\( L \). A formula
\[
	\exists \vec y \enskip \bigl(
		P_1(\vec x, \vec y)
		\wedge
		\ldots
		\wedge
		P_n(\vec x, \vec y)
	\bigr)
\]
for elementary formulae
\( P_i(\vec x, \vec y) \) is called
\textit{Diophantine} (also
\textit{positive-primitive} or
\textit{regular}) in the variables
\( \vec x \). Finally, a subset
\( X \subseteq M^n \) is called
\textit{Diophantine} with respect to parameters
\( \vec a \in M^m \) if
\[
	X
	=
	\{
		\vec x \in M^n
		\mid
		\varphi(\vec x, \vec a)
	\},
\]
where
\( \varphi \) is a Diophantine formula in the variables
\( \vec x \) and
\( \vec a \).

A closely related is the notion of
\textit{e-interpretation}. An e-interpretation of a model
\( M \) in a model
\( N \) (of possibly another language) is a surjection
\( X \to M \) for some Diophantine
\( X \subseteq N^n \) such that the lifts of all relations on
\( M \) to
\( X \) from its language are Diophantine (including the equivalence relation lifting the equality), as well as lifts of all graphs of operations on
\( M \). For example, if
\( K \) is a commutative ring and
\( G \) is a finitely presented affine group scheme over
\( K \), then the group
\( G(K) \) is e-interpretable in the ring
\( K \).

Suppose that
\( M \) and
\( L \) are enumerated by
\( \mathbb N \) or its finite subsets. The
\textit{Diophantine problem}
\( \mathcal D(M) \) is
\textit{decidable} if the non-emptiness of generic Diophantine set
\(
	\{
		\vec x
		\mid
		\varphi(\vec x, \vec a)
	\}
\) may be checked by an algorithm knowing only the codes of
\( \varphi \) and
\( \vec a \) (note that the set of Diophantine formulae is decidable). For example, the Diophantine problem
\( \mathcal D(\mathbb Z) \) (for the language of rings and a natural enumeration) is undecidable by famous Matiyasevich's theorem. On the other hand,
\( \mathcal D(\mathbb Q^{\mathrm{alg}}) \) is decidable (again for the language of rings and a natural enumeration) because every algebraically closed field admits quantifier elimination.

Now let
\( K \) be a commutative ring and
\( G \) be a reductive group scheme over
\( K \) in the sense of
\cite{sga3}. If
\( G \) is splits, i.e. it is a Chevalley--Demazure group scheme, simple, and of rank at least
\( 2 \), then by the main result of Elena Bunina, Alexey Myasnikov, and Eugene Plotkin
\cite{dioph-chev} the ring
\( K \) and the group
\( G(K) \) are e-interpretable in each other and the Diophantine problems
\( \mathcal D(K) \) and
\( \mathcal D(G(K)) \) are equivalent. More precisely, if
\( K \) is countable with fixed enumeration, then the Diophantine problems are reducible to each other by explicit algorithms. See also
\cite{%
	avn-lub-mei,%
	avn-mei,%
	bei-mik,%
	belegradek,%
	bun-uni,%
	bun-che-fie,%
	bun-che-loc,%
	bun-bii,%
	bun-mik,%
	kra-roh-ten,%
	maltsev,%
	mya-soh-tri,%
	mya-soh-lin,%
	mya-soh-cla,%
	segal-tent%
} for other results on model theory of linear groups over rings, including e-interpretability and first order interpretability.

We are going to generalize this result to isotropic
\( G \), not necessarily split. There are several possible definitions of such group schemes, for example, in Victor Petrov and Anastasia Stavrova's paper
\cite{iso-ele-nor} the term
\textit{isotropic} means that there exists a strictly proper parabolic subgroup. There is also a more general notion of
\textit{locally isotropic} reductive group schemes
\cite{loc-iso-ele}, namely, that the isotropicity condition holds locally in the Zariski topology. In this paper we impose an even stronger condition than Petrov and Stavrova, but for local rings all these definitions coincide. We plan to cover the locally isotropic case in a sequel paper using localization technique from
\cite{loc-iso-ele}. Of course, we actually need not only that
\( G \) is isotropic, but also that its suitable ``isotropic rank'' is at least
\( 2 \).

We follow the general strategy of Bunina -- Myasnikov -- Plotkin. Namely, root subgroups of
\( G(K) \) turn out to be Diophantine, see theorem
\ref{diophantine} below. This allows us to construct e-interpretation of the base ring
\( K \) in the group
\( G(K) \), this is theorem
\ref{e-interpret}. Finally, since
\( K \) and
\( G(K) \) e-interpret each other, their Diophantine problems are equivalent (theorem
\ref{dioph-prbl}). As an application of preparatory technical results we also find the centralizer and the normalizer of all root subgroups (and even their finite subsets) together, this is a generalization of
\cite{center-che, center-iso}.

For example, it follows that the Diophantine problem for
\( \mathrm{SO}(E_8 \perp H \perp H) \) is unsolvable, where
\( E_8 \) is the lattice spanned by the root system of type
\( \mathsf E_8 \) and
\( H \) is a hyperbolic plane over
\( \mathbb Z \).

The author wants to thank Eugene Plotkin for the problem statement and great motivation.

\section{Isotropic reductive groups}

In this paper all root systems are crystallographic, but not necessarily reduced, i.e. we allow components of type
\( \mathsf{BC}_\ell \). Let
\( K \) be a non-zero unital commutative ring and
\( \widetilde \Phi \) be a reduced root system. Choose a weight lattice
\[
	\textstyle
	\mathbb Z \widetilde \Phi
	\leq
	\Lambda
	\leq
	\bigl\{
		x
		\mid
		\forall \alpha \in \widetilde \Phi \enskip
			2 \frac{x \cdot \alpha}{\alpha \cdot \alpha}
			\in
			\mathbb Z
	\bigr\}
\]
and consider the Chevalley -- Demazure group scheme
\( \mathrm G^\Lambda(\widetilde \Phi, {-}) \) over
\( K \). We denote its standard maximal torus by
\( \widetilde T \) and its standard root subgroups by
\( \widetilde U_\alpha \) for
\( \alpha \in \widetilde \Phi \). Now choose
\(
	u
	\colon
	\widetilde \Phi \sqcup \{ 0 \}
	\to
	\Phi \sqcup \{ 0 \}
\) a surjective map arising from one of irreducible Tits indices (see below) and
\[
	L
	=
	\bigl\langle
		\widetilde T, \widetilde U_\alpha
		\mid
		u(\alpha) = 0
	\bigr\rangle,
	\quad
	U_\alpha
	=
	\bigl\langle
		\widetilde U_{\alpha'}
		\mid
		u(\alpha') \in \{ \alpha, 2 \alpha \}
	\bigr\rangle
\]
be subgroup sheaves of
\( \mathrm G^\Lambda(\widetilde \Phi, {-}) \), they are closed subgroups. More precisely,
\( L \) is split reductive with a maximal torus
\( \widetilde T \) and
\( U_\alpha \) are unipotent in a suitable sense.

In this paper we say that a group scheme
\( G \) together with a family of subgroup schemes
\( L, U_\alpha \leq G \) for
\( \alpha \in \Phi \) is an
\textit{isotropic reductive group scheme} if fppf locally it is isomorphic to
\( \mathrm G^\Lambda(\widetilde \Phi, {-}) \) with the standard subgroups constructed by a map
\(
	u
	\colon
	\widetilde \Phi \sqcup \{ 0 \}
	\to
	\Phi \sqcup \{ 0 \}
\). We also impose an additional condition that
\( G \) has
\textit{Weyl elements} locally in Zariski topology in the sense that locally there are
\(
	w_\alpha
	\in
	U_\alpha(K)\, U_{-\alpha}(K)\, U_\alpha(K)
\) such that
\(
	\up{w_\alpha}{U_\beta}
	=
	U_{
		\beta
		-
		2
		\frac {\alpha \cdot \beta}{\alpha \cdot \alpha}
		\alpha
	}
\) as group subschemes (we conjecture that this condition always holds). An isotropic group scheme
\( G \)
\textit{splits} if it is isomorphic to
\( \mathrm G^\Lambda(\widetilde \Phi, {-}) \) with the standard subgroups.

This class of group schemes is interesting because every simple reductive group scheme
\( G \) over a local ring
\( K \) is anisotropic or isotropic in the above sense in a canonical way up to conjugation, where
\[
	L
	\rtimes
	\bigl\langle
		U_\alpha
		\mid
		\alpha \in \Phi^+
	\bigr\rangle
\]
is a minimal parabolic subgroup, see
\cite[XXVI]{sga3} and
\cite{index-loc} for details. Moreover, if
\( G \) is a locally isotropic simple reductive group scheme over arbitrary ring
\( K \) and
\( \mathfrak p \leqt K \) is a prime ideal, then the isotropic structure of
\( G_{\mathfrak p} \) is well defined on
\( G_s \) for sufficiently large
\( s \in K \setminus \mathfrak p \).

Recall
\cite{index-field} that an irreducible Tits index
\( (\widetilde \Phi, \Gamma, J) \) consists of a reduced irreducible root system
\( \widetilde \Phi \), a subgroup
\( \Gamma \) of its group of outer automorphisms (i.e. the automorphism group of its Dynkin diagram), and a
\( \Gamma \)-invariant subset
\( J \) of vertices of the Dynkin diagram from the explicit list
\cite[table II]{index-field} or
\cite[appendix]{index-loc}. Then
\( \Phi \) is the image of
\( \widetilde \Phi \) in the factor-space of the ambient vector space by
\( \Gamma \) and the span of basic roots not in
\( J \), the rank of
\( \Phi \) is the number of
\( \Gamma \)-orbits in
\( J \). An index is usually denoted as
\( \up g {\mathsf X_{n, r}^t} \), where
\( g = |\Gamma| \),
\( \mathsf X_n \) is the type of
\( \widetilde \Phi \),
\( r \) is the rank of
\( \Phi \), and
\( t \) is an additional parameter. We are interested only in isotropic Tits indices, i.e. with
\( r \neq 0 \). An
\textit{outer automorphism group} of an irreducible Tits index is the group of such automorphisms
\( f \) of the root system
\( \widetilde \Phi \) that
\( f \) preserves the basis and
\( u \circ f = u \).

Now we give the list of all isotropic irreducible Tits indices. It is divided in two parts,
\textit{classical} indices (forming infinite families) and finitely many
\textit{exceptional} ones. The numeration of simple roots in all Dynkin diagrams follows Bourbaki
\cite{bourbaki}. For classical root systems we use the conventions
\(
	\mathsf A_{-1}
	=
	\mathsf A_0
	=
	\mathsf B_0
	=
	\mathsf C_0
	=
	\mathsf{BC}_0
	=
	\mathsf D_0
	=
	\mathsf D_1
	=
	\varnothing
\) and
\( \mathsf D_2 = 2 \mathsf A_1 \), as well as the exceptional isomorphisms
\( \mathsf A_1 = \mathsf B_1 = \mathsf C_1 \),
\( \mathsf B_2 = \mathsf C_2 \),
\( \mathsf A_3 = \mathsf D_3 \).

\begin{itemize}

	\item
	\( \up 1 {\mathsf A_{n, r}^{(d)}} \), where
	\( d (r + 1) = n + 1 \),
	\( d \geq 1 \), and
	\( n \geq 1 \). Here
	\( \widetilde \Phi = \mathsf A_n \),
	\( J = \{ d, 2 d, \ldots, r d \} \), and
	\( \Phi = \mathsf A_r \). The kernel of
	\( u \) has the type
	\( (r + 1) A_{d - 1} \) and
	\( |u^{-1}(\alpha)| = d^2 \) for every root
	\( \alpha \). The outer automorphism group is trivial for
	\( r \geq 2 \) or
	\( \up 1 {\mathsf A_{1, 1}^{(1)}} \), it is cyclic of order
	\( 2 \) for
	\( r = 1 \) and
	\( d \geq 2 \).
	
	Such split isotropic structure on
	\( \mathbb{SL}_{n + 1} \) corresponds to the block decomposition by
	\( d \times d \) blocks,
	\( U_\alpha \) are the groups subschemes of matrices differing from
	\( 1 \) only in one non-diagonal block, and
	\( L \) is the group subscheme of block diagonal matrices.

	\item
	\( \up 2 {\mathsf A_{n, r}^{(d)}} \), where
	\( d \mid n + 1 \),
	\( 2 r d \leq n + 1 \),
	\( d \geq 1 \),
	\( r \geq 1 \), and
	\( n \geq 2 \). Here
	\( \widetilde \Phi = \mathsf A_n \),
	\[
		J
		=
		\{
			d,
			2 d,
			\ldots,
			r d
		\}
		\sqcup
		\{
			n + 1 - r d,
			n + 1 - (r - 1) d,
			\ldots,
			n + 1 - d
		\}
	\]
	for
	\( 2 r d \leq n \) and
	\( J = \{ d, 2 d, \ldots, n + 1 - d \} \) for
	\( 2 r d = n + 1 \),
	\( \Phi = \mathsf{BC}_r \) for
	\( 2 r d \leq n \) and
	\( \Phi = \mathsf C_r \) for
	\( 2 r d = n + 1 \). The kernel of
	\( u \) has the type,
	\(
		2 r \mathsf A_{d - 1}
		+
		\mathsf A_{n - 2 r d}
	\),
	\( |u^{-1}(\alpha)| = 2 d (n + 1 - 2 r d) \) for ultrashort
	\( \alpha \),
	\( |u^{-1}(\alpha)| = 2 d^2 \) for short
	\( \alpha \),
	\( |u^{-1}(\alpha)| = d^2 \) for long
	\( \alpha \). The outer automorphism group is cyclic of order
	\( 2 \).
	
	Such split isotropic structure on
	\( \mathbb{SL}_{n + 1} \) corresponds to the block decomposition by
	\( r \) strips of width
	\( d \) from each side and, possibly, the central strip of width
	\( n + 1 - 2 r d \), both vertical and horizontal. We number these strips by integers from
	\( - r \) to
	\( r \),
	\( 0 \) corresponds to the central strip. The group subscheme
	\( L \) consists of block-diagonal matrices. If
	\( \alpha \) is short, then
	\( U_\alpha \) consists of matrices differing from
	\( 1 \) only in two blocks with indices
	\( (i, j) \) and
	\( (- j, - i) \) for some
	\( 0 \neq |i| \neq |j| \neq 0 \). If
	\( \alpha \) is long, then
	\( U_\alpha \) consists of matrices differing from
	\( 1 \) only in the block with the indices
	\( (- i, i) \) for some
	\( i \neq 0 \). Finally, if
	\( \alpha \) is ultrashort, then
	\( U_\alpha \) consists of matrices differing from
	\( 1 \) only in the blocks with the indices
	\( (0, i) \),
	\( (- i, i) \), and
	\( (- i, 0) \) for some
	\( i \neq 0\).

	\item
	\( \mathsf B_{n, r} \), where
	\( 1 \leq r \leq n \). Here
	\( \widetilde \Phi = \mathsf B_n \),
	\( J = \{ 1, 2, \ldots, r \} \), and
	\( \Phi = \mathsf B_r \). The kernel of
	\( u \) had the type
	\( \mathsf B_{n - r} \),
	\( |u^{-1}(\alpha)| = 2 n + 1 - 2 r \) for short
	\( \alpha \), and
	\( |u^{-1}(\alpha)| = 1 \) for long
	\( \alpha \). The outer automorphism group is trivial.
	
	Such split isotropic structure on
	\( \mathbb{SO}_{2 n + 1} \) is induced from the block decomposition
	\( \up 2 {\mathsf A_{2 n + 1, r}^{(1)}} \) on the standard representation, i.e.
	\( \mathbb{SO}_{2 n + 1} \leq \mathbb{SL}_{2 n + 1} \) is the stabilizer of the quadratic form
	\[
		q(\vec x)
		=
		x_1 x_{2 n + 1}
		+
		x_2 x_{2 n}
		+
		\ldots
		+
		x_n x_{n + 1}
		+
		x_{n + 1}^2.
	\]

	\item
	\( \mathsf C_{n, r}^{(d)} \), where
	\( d = 2^k \mid 2 n \),
	\( r d \leq n \),
	\( d \geq 1 \),
	\( r \geq 1 \), and
	\( n = r \) in the case
	\( d = 1 \). Here
	\( \widetilde \Phi = \mathsf C_n \),
	\( J = \{ d, 2 d, \ldots, r d \} \),
	\( \Phi = \mathsf{BC}_r \) for
	\( r d \leq n - 1 \) and
	\( \Phi = \mathsf C_r \) for
	\( r d = n \). The kernel of
	\( u \) has the type
	\(
		u^{-1}(0)
		=
		r \mathsf A_{d - 1}
		+
		\mathsf C_{n - r d}
	\),
	\( |u^{-1}(\alpha)| = 2 d (n - r d) \) for ultrashort
	\( \alpha \),
	\( |u^{-1}(\alpha)| = d^2 \) for short
	\( \alpha \),
	\( |u^{-1}(\alpha)| = \frac {d (d + 1)} 2 \) for long
	\( \alpha \). The outer automorphism group is trivial.
	
	Such split isotropic structure on
	\( \mathbb{Sp}_{2 n} \) is induced from the block decomposition
	\( \up 2 {\mathsf A_{2 n, r}^{(1)}} \) on the standard representation, i.e.
	\( \mathbb{Sp}_{2 n} \leq \mathbb{GL}_{2 n} \) is the stabilizer of the symplectic form
	\[
		B(\vec x, \vec y)
		=
		x_1 y_{2 n}
		+
		x_2 y_{2 n - 1}
		+
		\ldots
		+
		x_n y_{n + 1}
		-
		x_{n + 1} y_n
		-
		x_{n + 2} y_{n - 1}
		-
		\ldots
		-
		x_{2 n} y_1.
	\]

	\item
	\( \up 1 {\mathsf D_{n, r}^{(d)}} \), where
	\( d = 2^k \mid 2 n \),
	\( r d \leq n \),
	\( r d \neq n - 1 \),
	\( d \geq 1 \),
	\( r \geq 1 \),
	\( n \geq 3 \). Here
	\( \widetilde \Phi = D_n \),
	\( J = \{ d, 2 d, \ldots, r d \} \),
	\( \Phi = \mathsf{BC}_r \) for
	\( r d \leq n - 2 \) and
	\( d \geq 2 \),
	\( \Phi = \mathsf B_r \) for
	\( r d \leq n - 2 \) and
	\( d = 1 \),
	\( \Phi = \mathsf C_r \) for
	\( r d = n \) and
	\( d \geq 2 \),
	\( \Phi = \mathsf D_r \) for
	\( r d = n \) and
	\( d = 1 \). The kernel of
	\( u \) has the type
	\(
		r \mathsf A_{d - 1}
		+
		\mathsf D_{n - r d}
	\),
	\( |u^{-1}(\alpha)| = 2 d (n - r d) \) for ultrashort
	\( \alpha \),
	\( |u^{-1}(\alpha)| = d^2 \) for short
	\( \alpha \),
	\( |u^{-1}(\alpha)| = \frac {d (d - 1)} 2 \) for long
	\( \alpha \) (if
	\( d = 1 \), then
	\( |u^{-1}(\alpha)| = 2 d (n - r d) \) for short
	\( \alpha \) and
	\( |u^{-1}(\alpha)| = d^2 \) for long
	\( \alpha \)). The outer automorphism group is trivial for
	\( r d = n \) except the cases
	\( \up 1 {\mathsf D_{4, 2}^{(2)}} \) and
	\( \up 1 {\mathsf D_{4, 1}^{(4)}} \); it is cyclic of order two for
	\( \up 1 {\mathsf D_{4, 2}^{(2)}} \),
	\( \up 1 {\mathsf D_{4, 1}^{(4)}} \), and for
	\( r d \leq n - 2 \) except the case
	\( \up 1 {\mathsf D_{4, 1}^{(2)}} \); it is the whole triality group of order
	\( 6 \) in the case
	\( \up 1 {\mathsf D_{4, 1}^{(2)}} \).
	
	Such split isotropic structure on
	\( \mathbb{SO}_{2 n} \) is induced from the block decomposition
	\( \up 2 {\mathsf A_{2 n, r}^{(1)}} \) on the standard representation, i.e.
	\( \mathbb{SO}_{2 n} \leq \mathbb{GL}_{2 n} \) is the stabilizer of the quadratic form
	\[
		q(\vec x)
		=
		x_1 x_{2 n + 1}
		+
		x_2 x_{2 n}
		+
		\ldots
		+
		x_n x_{n + 1}
	\]
	and the Dickson invariant (or the determinant if
	\( 2 \) is invertible).

	\item
	\( \up 2 {\mathsf D_{n, r}^{(d)}} \), where
	\( d = 2^k \mid 2 n \),
	\( r d \leq n - 1 \),
	\( d \geq 1 \),
	\( r \geq 1 \),
	\( n \geq 3 \). Here
	\( \widetilde \Phi = \mathsf D_n \),
	\( J = \{ d, 2 d, \ldots, r d\} \) for
	\( r d < n - 2 \) and
	\( J = \{ d, 2 d, \ldots, (r - 1) d, n - 1, n \} \) for
	\( r d = n - 1 \),
	\( \Phi = \mathsf{BC}_r \) for
	\( d \geq 2 \) and
	\( \Phi = \mathsf B_r \) for
	\( d = 1 \). The kernel of
	\( u \) has the type
	\(
		r \mathsf A_{d - 1}
		+
		\mathsf D_{n - r d}
	\),
	\( |u^{-1}(\alpha)| = 2 d (n - r d) \) for ultrashort
	\( \alpha \),
	\( |u^{-1}(\alpha)| = d^2 \) for short
	\( \alpha \),
	\( |u^{-1}(\alpha)| = \frac {d (d - 1)} 2 \) for long
	\( \alpha \) (if
	\( d = 1 \), then
	\( |u^{-1}(\alpha)| = 2 d (n - r d) \) for short
	\( \alpha \) an
	\( |u^{-1}(\alpha)| = d^2 \) for long
	\( \alpha \)). The outer automorphism group is cyclic of order
	\( 2 \) except the case
	\( \up 2 {\mathsf D_{4, 1}^{(2)}} \) and it is the triality group of order
	\( 6 \) in this exceptional case.
	
	Such split isotropic structure on
	\( \mathbb{SO}_{2 n} \) is induced from the block decomposition
	\( \up 2 {\mathsf A_{2 n, r}^{(1)}} \) on the standard representation as in the previous case, independently of whether
	\( d = 1 \) or not.

\end{itemize}

Actually, all classical isotropic reductive group schemes can be described in terms of involution algebras
\cite{twi-for}.

Exceptional isotropic irreducible Tits indices are given in the following table. Sizes and types of preimages are given by increasing the root length. The indices of labels in
\( J \) denote the lengths of the corresponding roots in
\( \Phi \), namely, \textbf{l}ong, \textbf{s}hort, and \textbf{u}ltra\-\textbf{s}hort ones. When working with these group schemes it is useful to use explicit root diagrams of all exceptional root systems
\cite{atlas}, such diagrams split into the subdiagram
\( u^{-1}(0) \) and blocks
\( u^{-1}(\alpha) \) connected by edges and outer automorphisms of the Tits index.

\begin{center}
	\begin{tabular}{|c|c|c|c|c|c|c|}
		\hline
		Tits index
		& \( \widetilde \Phi \)
		& \( J \)
		& \( \Phi \)
		& \( u^{-1}(0) \)
		& \( |u^{-1}(\alpha)| \)
		& \( |\mathrm{Out}| \)
		\\ \hline \hline
		\( \mathsf E_{7, 1}^{78} \)
		& \( \mathsf E_7 \)
		& \( \{ 7 \} \)
		& \( \mathsf A_1 \)
		& \( \mathsf E_6 \)
		& \( 27 \)
		& \( 1 \)
		\\ \hline
		\( \up 3 {\mathsf D_{4, 1}^9} \),
		\( \up 6 {\mathsf D_{4, 1}^9} \)
		& \( \mathsf D_4 \)
		& \( \{ 2 \} \)
		& \( \mathsf{BC}_1 \)
		& \( 3 \mathsf A_1 \)
		& \( 8 \),
		\( 1 \)
		& \( 6 \)
		\\ \hline
		\( \up 2 {\mathsf E_{6, 1}^{35}} \)
		& \( \mathsf E_6 \)
		& \( \{ 2 \} \)
		& \( \mathsf{BC}_1 \)
		& \( \mathsf A_5 \)
		& \( 20 \),
		\( 1 \)
		& \( 2 \)
		\\ \hline
		\( \mathsf E_{7, 1}^{66} \)
		& \( \mathsf E_7 \)
		& \( \{ 1 \} \)
		& \( \mathsf{BC}_1 \)
		& \( \mathsf D_6 \)
		& \( 32 \),
		\( 1 \)
		& \( 1 \)
		\\ \hline
		\( \mathsf E_{8, 1}^{133} \)
		& \( \mathsf E_8 \)
		& \( \{ 8 \} \)
		& \( \mathsf{BC}_1 \)
		& \( \mathsf E_7 \)
		& \( 56 \),
		\( 1 \)
		& \( 1 \)
		\\ \hline
		\( \mathsf F_{4, 1}^{21} \)
		& \( \mathsf F_4 \)
		& \( \{ 4 \} \)
		& \( \mathsf{BC}_1 \)
		& \( \mathsf B_3 \)
		& \( 8 \),
		\( 7 \)
		& \( 1 \)
		\\ \hline
		\( \up 2 {\mathsf E_{6, 1}^{29}} \)
		& \( \mathsf E_6 \)
		& \( \{ 1, 6 \} \)
		& \( \mathsf{BC}_1 \)
		& \( \mathsf D_4 \)
		& \( 16 \),
		\( 8 \)
		& \( 2 \)
		\\ \hline
		\( \mathsf E_{7, 1}^{48} \)
		& \( \mathsf E_7 \)
		& \( \{ 6 \} \)
		& \( \mathsf{BC}_1 \)
		& \( \mathsf A_1 + \mathsf D_5 \)
		& \( 32 \),
		\( 10 \)
		& \( 1 \)
		\\ \hline
		\( \mathsf E_{8, 1}^{91} \)
		& \( \mathsf E_8 \)
		& \( \{ 1 \} \)
		& \( \mathsf{BC}_1 \)
		& \( \mathsf D_7 \)
		& \( 64 \),
		\( 14 \)
		& \( 1 \)
		\\ \hline \hline
		\( \up 1 {\mathsf E_{6, 2}^{28}} \)
		& \( \mathsf E_6 \)
		& \( \{ 1, 6 \} \)
		& \( \mathsf A_2 \)
		& \( \mathsf D_4 \)
		& \( 8 \)
		& \( 1 \)
		\\ \hline
		\( \mathsf G_{2, 2}^0 \)
		& \( \mathsf G_2 \)
		& \( \{ 1_{\mathrm s}, 2_{\mathrm l} \} \)
		& \( \mathsf G_2 \)
		& \( \varnothing \)
		& \( 1 \),
		\( 1 \)
		& \( 1 \)
		\\ \hline
		\( \up 3 {\mathsf D_{4, 2}^2} \),
		\( \up 6 {\mathsf D_{4, 2}^2} \)
		& \( \mathsf D_4 \)
		& \(
			\{
				1_{\mathrm s},
				2_{\mathrm l},
				3_{\mathrm s},
				4_{\mathrm s}
			\}
		\)
		& \( \mathsf G_2 \)
		& \( \varnothing \)
		& \( 3 \),
		 \( 1 \)
		& \( 6 \)
		\\ \hline
		\( \up 1 {\mathsf E_{6, 2}^{16}} \),
		\( \up 2 {\mathsf E_{6, 2}^{16''}} \)
		& \( \mathsf E_6 \)
		& \( \{ 2_{\mathrm l}, 4_{\mathrm s} \} \)
		& \( \mathsf G_2 \)
		& \( 2 \mathsf A_2 \)
		& \( 9 \),
		\( 1 \)
		& \( 2 \)
		\\ \hline
		\( \mathsf E_{8, 2}^{78} \)
		& \( \mathsf E_8 \)
		& \( \{ 7_{\mathrm s}, 8_{\mathrm l} \} \)
		& \( \mathsf G_2 \)
		& \( \mathsf E_6 \)
		& \( 27 \),
		\( 1 \)
		& \( 1 \)
		\\ \hline
		\( \up 2 {\mathsf E_{6, 2}^{16'}} \)
		& \( \mathsf E_6 \)
		& \(
			\{
				1_{\mathrm{us}},
				2_{\mathrm s},
				6_{\mathrm{us}}
			\}
		\)
		& \( \mathsf{BC}_2 \)
		& \( \mathsf A_3 \)
		& \( 8 \),
		\( 6 \),
		\( 1 \)
		& \( 2 \)
		\\ \hline
		\( \mathsf E_{7, 2}^{31} \)
		& \( \mathsf E_7 \)
		& \( \{ 1_{\mathrm s}, 6_{\mathrm{us}} \} \)
		& \( \mathsf{BC}_2 \)
		& \( \mathsf A_1 + \mathsf D_4 \)
		& \( 16 \),
		\( 8 \),
		\( 1 \)
		& \( 1 \)
		\\ \hline
		\( \mathsf E_{8, 2}^{66} \)
		& \( \mathsf E_8 \)
		& \( \{ 1_{\mathrm{us}}, 8_{\mathrm s} \} \)
		& \( \mathsf{BC}_2 \)
		& \( \mathsf D_6 \)
		& \( 32 \),
		\( 12 \),
		\( 1 \)
		& \( 1 \)
		\\ \hline \hline
		\( \mathsf E_{7, 3}^{28} \)
		& \( \mathsf E_7 \)
		& \(
			\{
				1_{\mathrm s},
				6_{\mathrm s},
				7_{\mathrm l}
			\}
		\)
		& \( \mathsf C_3 \)
		& \( \mathsf D_4 \)
		& \( 8 \),
		\( 1 \)
		& \( 1 \)
		\\ \hline
		\( \mathsf F_{4, 4}^0 \)
		& \( \mathsf F_4 \)
		& \(
			\{
				1_{\mathrm l},
				2_{\mathrm l},
				3_{\mathrm s},
				4_{\mathrm s}
			\}
		\)
		& \( \mathsf F_4 \)
		& \( \varnothing \)
		& \( 1 \),
		\( 1 \)
		& \( 1 \)
		\\ \hline
		\( \up 2 {\mathsf E_{6, 4}^2} \)
		& \( \mathsf E_6 \)
		& \(
			\{
				1_{\mathrm s},
				2_{\mathrm l},
				3_{\mathrm s},
				4_{\mathrm l},
				5_{\mathrm s},
				6_{\mathrm s}
			\}
		\)
		& \( \mathsf F_4 \)
		& \( \varnothing \)
		& \( 2 \),
		\( 1 \)
		& \( 2 \)
		\\ \hline
		\( \mathsf E_{7, 4}^9 \)
		& \( \mathsf E_7 \)
		& \(
			\{
				1_{\mathrm l},
				3_{\mathrm l},
				4_{\mathrm s},
				6_{\mathrm s}
			\}
		\)
		& \( \mathsf F_4 \)
		& \( 3 \mathsf A_1 \)
		& \( 4 \),
		\( 1 \)
		& \( 1 \)
		\\ \hline
		\( \mathsf E_{8, 4}^{28} \)
		& \( \mathsf E_8 \)
		& \(
			\{
				1_{\mathrm s},
				6_{\mathrm s},
				7_{\mathrm l},
				8_{\mathrm l}
			\}
		\)
		& \( \mathsf F_4 \)
		& \( \mathsf D_4 \)
		& \( 8 \),
		\( 1 \)
		& \( 1 \)
		\\ \hline
		\( \up 1 {\mathsf E_{6, 6}^0} \)
		& \( \mathsf E_6 \)
		& \( \{ 1, 2, 3, 4, 5, 6 \} \)
		& \( \mathsf E_6 \)
		& \( \varnothing \)
		& \( 1 \)
		& \( 1 \)
		\\ \hline
		\( \mathsf E_{7, 7}^0 \)
		& \( \mathsf E_7 \)
		& \( \{ 1, 2, 3, 4, 5, 6, 7 \} \)
		& \( \mathsf E_7 \)
		& \( \varnothing \)
		& \( 1 \)
		& \( 1 \)
		\\ \hline
		\( \mathsf E_{8, 8}^0 \)
		& \( \mathsf E_8 \)
		& \( \{ 1, 2, 3, 4, 5, 6, 7, 8 \} \)
		& \( \mathsf E_8 \)
		& \( \varnothing \)
		& \( 1 \)
		& \( 1 \)
		\\ \hline
	\end{tabular}
\end{center}

Finally, we give the complete list of isomorphisms and equivalences between isotropic irreducible Tits indices. Two such indices are called
\textit{isomorphic} (denoted by
\( \cong \)) if they differ by an automorphism of the Dynkin diagram of
\( \widetilde \Phi \) and, possibly, by its renaming via an exceptional isomorphism. We say that two indices are
\textit{equivalent} (denoted by
\( \sim \)) if the induced maps
\( u \) differ by isomorphisms between the corresponding root systems
\( \widetilde \Phi \) and
\( \Phi \). Equivalent Tits indices are indistinguishable for our purposes.

\begin{align*}
	\up 1 {\mathsf A_{1, 1}^{(1)}}
	&\cong
	\mathsf B_{1, 1}
	\cong
	\mathsf C_{1, 1}^{(1)},
	&
	\up 1 {\mathsf D_{4, 1}^{(1)}}
	&\cong
	\up 1 {\mathsf D_{4, 1}^{(4)}}
	\sim
	\up 2 {\mathsf D_{4, 1}^{(1)}},
	\\
	\mathsf B_{2, 1}
	&\cong
	\mathsf C_{2, 1}^{(2)},
	&
	\up 1 {\mathsf D_{4, 1}^{(2)}}
	&\sim
	\up 2 {\mathsf D_{4, 1}^{(2)}}
	\sim
	\up 3 {\mathsf D_{4, 1}^9}
	\sim
	\up 6 {\mathsf D_{4, 1}^9},
	\\
	\mathsf B_{2, 2}
	&\cong
	\mathsf C_{2, 2}^{(1)},
	&
	\up 1 {\mathsf D_{4, 2}^{(1)}}
	&\cong
	\up 1 {\mathsf D_{4, 2}^{(2)}}
	\sim
	\up 2 {\mathsf D_{4, 2}^{(1)}},
	\\
	\up 2 {\mathsf A_{3, 1}^{(1)}}
	&\cong
	\up 2 {\mathsf D_{3, 1}^{(2)}},
	&
	\up 3 {\mathsf D_{4, 2}^2}
	&\sim
	\up 6 {\mathsf D_{4, 2}^2},
	\\
	\up 2 {\mathsf A_{3, 1}^{(2)}}
	&\cong
	\up 2 {\mathsf D_{3, 1}^{(1)}}
	\sim
	\up 1 {\mathsf A_{3, 1}^{(2)}}
	\cong
	\up 1 {\mathsf D_{3, 1}^{(1)}},
	&
	\up 1 {\mathsf E_{6, 2}^{16}}
	&\sim
	\up 2 {\mathsf E_{6, 2}^{16''}},
	\\
	\up 2 {\mathsf A_{3, 2}^{(1)}}
	&\cong
	\up 2 {\mathsf D_{3, 2}^{(1)}},
	&
	\up 1 {\mathsf A_{n, 1}^{(d)}}
	&\sim
	\up 2 {\mathsf A_{n, 1}^{(d)}}
	\text{ for }
	d \geq 3,\,
	n \geq 5,\,
	2 d = n + 1,
	\\
	\up 1 {\mathsf A_{3, 3}^{(1)}}
	&\cong
	\up 1 {\mathsf D_{3, 3}^{(1)}},
	&
	\up 1 {\mathsf D_{n, r}^{(d)}}
	&\sim
	\up 2 {\mathsf D_{n, r}^{(d)}}
	\text{ for }
	d = 2^k \mid 2 n,\,
	d \geq 1,\,
	r \geq 1,\,
	n \geq 5,\,
	r d \leq n - 2.
\end{align*}

\section{Technical lemmas}

Recall that a subset
\( \Sigma \subseteq \Phi \) is called
\textit{closed} if
\( (\Sigma + \Sigma) \cap \Phi \subseteq \Sigma \). A closed subset
\( \Sigma \) is
\begin{itemize}

	\item
	\textit{unipotent} if it is contained in an open half-space (equivalently, if it does not contain opposite roots);

	\item
	\textit{closed root subsystem} if
	\( \Sigma = -\Sigma \), so it is a root system itself;

	\item
	\textit{parabolic} if
	\( \Phi = \Sigma \cup (-\Sigma) \);

	\item
	\textit{saturated} if
	\( \Sigma = \Phi \cap \mathbb R_{\geq 0} \Sigma \).

\end{itemize}

Note that any parabolic set is saturated (and every saturated set is an intersection of parabolic ones), but
\( \mathsf A_2 \subseteq \mathsf G_2 \) and
\( \mathsf C_2 \subseteq \mathsf{BC}_2 \) are not saturated. Any closed subset
\( \Sigma \) admits a unique decomposition
\(
	\Sigma
	=
	\Sigma_{\mathrm r} \sqcup \Sigma_{\mathrm u}
\), where
\( \Sigma_{\mathrm r} = \Sigma \cap (-\Sigma) \) is a closed root subsystem and
\( \Sigma_{\mathrm u} = \Sigma \setminus (-\Sigma) \) is a unipotent set. Conversely, if
\( \Sigma_{\mathrm r} \) is a closed root subsystem,
\( \Sigma_{\mathrm u} \) is a unipotent set disjoint with
\( \Sigma_{\mathrm r} \), and
\(
	(\Sigma_{\mathrm r} + \Sigma_{\mathrm u}) \cap \Phi
	\subseteq
	\Sigma_{\mathrm u}
\), then
\( \Sigma_{\mathrm r} \cup \Sigma_{\mathrm u} \) is a closed set. A closed set
\( \Sigma \) is parabolic if and only if there are
\( w \in \mathrm W(\Phi) \) and
\( J \subseteq \Delta \) such that
\( w \Sigma = (\Phi^+ + \mathbb Z J) \cap \Phi \). Here
\( \Phi^+ \subseteq \Phi \) is the set of positive roots and
\( \Delta \subseteq \Phi^+ \) is the set of basic roots. The classes of closed sets, unipotent sets, closed root subsystems, and saturated sets are closed under intersection.

If
\( \Sigma \subseteq \Phi \) is closed, then
\(
	u^{-1}(\Sigma \cup \{ 0 \})
	\subseteq
	\widetilde \Phi
\) is also closed. More precisely,
\begin{itemize}

	\item
	if
	\( \Sigma \) is unipotent, then
	\( u^{-1}(\Sigma) \) is unipotent;

	\item
	if
	\( \Sigma \) is a closed root subsystem, then
	\( u^{-1}(\Sigma \cup \{ 0 \}) \) is a closed root subsystem;

	\item
	if
	\( \Sigma \) is parabolic, then
	\( u^{-1}(\Sigma \cup \{ 0 \}) \) is parabolic;

	\item
	if
	\( \Sigma \) is saturated, then
	\( u^{-1}(\Sigma \cup \{ 0 \}) \) is saturated.

\end{itemize}
If
\( \Sigma \subseteq \Phi \) is a closed root subsystem, then the smallest closed set containing
\( u^{-1}(\Sigma) \) is also a closed root subsystem.

For any closed subset
\( \Sigma \subseteq \Phi \) let
\( G_\Sigma \leq G \) and
\( G^0_\Sigma \leq G \) be the group subsheaves generated by
\( \bigcup_{\alpha \in \Sigma} U_\alpha \) and
\( L \cup \bigcup_{\alpha \in \Sigma} U_\alpha \), they are actually smooth closed group subschemes with connected geometric fibers over
\( K \) and
\( G^0_\Sigma \) is of type (RC)
\cite[\S XXII.5]{sga3}. More precisely, if
\( \Sigma \) is a root subsystem, then
\( G_\Sigma \) and
\( G^0_\Sigma \) are reductive group subschemes (with induced isotropic structure). If
\( \Sigma \) is unipotent, then
\( G_\Sigma \) is a unipotent subgroup, i.e. it is an \'etale twisted form of some
\( \mathbb A^N \) as a scheme. Finally, in general
\(
	G_\Sigma
	=
	G_{\Sigma_{\mathrm r}}
	\rtimes
	G_{\Sigma_{\mathrm u}}
\) and
\(
	G^0_\Sigma
	=
	G^0_{\Sigma_{\mathrm r}}
	\rtimes
	G_{\Sigma_{\mathrm u}}
\). If
\( \Sigma \) is parabolic, then
\( G^0_\Sigma \) is a parabolic group subscheme. For any parabolic
\( \Sigma \subseteq \Phi \) the multiplication morphism
\[
	G_{-\Sigma_{\mathrm u}}
	\times
	G^0_{\Sigma_{\mathrm r}}
	\times
	G_{\Sigma_{\mathrm u}}
	\to
	G
\]
is an open scheme embedding (so for any closed
\( \Sigma \) such a morphism is just a scheme embedding). Also, for any unipotent
\( \Sigma \) the multiplication morphism
\(
	\prod_{\alpha \in \Sigma \setminus 2 \Sigma}
		G_\alpha
	\to
	G_\Sigma
\) is an isomorphism of schemes for any order of the factors.

\begin{lemma}
	\label{gauss}
	Suppose that
	\( K \) is semi-local. Then every isotropic reductive group
	\( G \) has a
	\textit{Gauss decomposition}
	\[
		G(K)
		=
		G_{\Phi^+}(K)\,
		G_{\Phi^-}(K)\,
		G_{\Phi^+}(K)\,
		L(K).
	\]
	Moreover, for any closed subset
	\( \Sigma \subseteq \Phi \) there is a decomposition
	\[
		G^0_\Sigma(K)
		=
		G_{\Sigma_{\mathrm r} \cap \Phi^+}(K)\,
		G_{\Sigma \cap \Phi^-}(K)\,
		G_{\Sigma \cap \Phi^+}(K)\,
		L(K).
	\]
\end{lemma}
\begin{proof}
	We may assume that
	\( \Spec(K) \) is connected. In this case there is a maximal isotropic structure
	\( (L', G'_\alpha)_{\alpha \in \Phi'} \) and a map
	\( v \colon \Phi' \to \Phi \cup \{ 0 \} \) such that
	\( \Phi \subseteq v(\Phi') \),
	\( G'_\alpha = G_{v^{-1}(\alpha)} \),
	\( L' = G_{\Phi \cap v^{-1}(0)} \) \cite[XXVI.7.4.2]{sga3}. Gauss decomposition holds for
	\( G \) with respect to a maximal isotropic structure by
	\cite[corollary XXVI.5.2]{sga3}, i.e.
	\[
		G(K)
		=
		G'_{\Phi^{\prime +}}(K)\,
		G'_{\Phi^{\prime -}}(K)\,
		G'_{\Phi^{\prime +}}(K)\,
		L'(K).
	\]
	It follows that
	\[
		G(K)
		=
		G_{\Phi^+}(K)\,
		G_{\Phi^-}(K)\,
		G_{\Phi^+}(K)\,
		L(K)
	\]
	and
	\[
		G^0_\Sigma(K)
		=
		G_{\Sigma \cap \Phi^+}(K)\,
		G_{\Sigma \cap \Phi^-}(K)\,
		G_{\Sigma \cap \Phi^+}(K)\,
		L(K)
	\]
	for any root subsystem
	\( \Sigma \). If
	\( \Sigma \subseteq \Phi \) is a closed subset, then
	\begin{align*}
		G^0_\Sigma(K)
		&=
		G_{\Sigma_{\mathrm r} \cap \Phi^+}(K)\,
		G_{\Sigma_{\mathrm r} \cap \Phi^-}(K)\,
		G_{\Sigma_{\mathrm u}}(K)\,
		G_{\Sigma_{\mathrm r} \cap \Phi^+}(K)\,
		L(K) \\
		&=
		G_{\Sigma_{\mathrm r} \cap \Phi^+}(K)\,
		G_{\Sigma \cap \Phi^-}(K)\,
		G_{\Sigma \cap \Phi^+}(K)\,
		L(K).\qedhere
	\end{align*}
\end{proof}

For any closed subsets
\( \Sigma, \Sigma' \subseteq \Phi \) the intersection
\( G^0_\Sigma \cap G^0_{\Sigma'} \) is a smooth closed subscheme and its fiberwise connected component is
\( G^0_{\Sigma \cap \Sigma'} \) by
\cite[proposition XXII.5.4.5]{sga3}.

\begin{lemma}
	\label{subgr-int}
	Let
	\( \Sigma, \Sigma' \subseteq \Phi \) be closed subsets. If
	\( \Sigma' \) is saturated, then
	\(
		G^0_\Sigma \cap G^0_{\Sigma'}
		=
		G^0_{\Sigma \cap \Sigma'}
	\). If
	\( \Sigma' \) is unipotent, then
	\(
		G^0_\Sigma \cap G_{\Sigma'}
		=
		G_{\Sigma} \cap G_{\Sigma'}
		=
		G_{\Sigma \cap \Sigma'}
	\).
\end{lemma}
\begin{proof}
	For the first claim it suffices to prove that
	\(
		G^0_\Sigma \cap G^0_{\Sigma'}
		=
		G^0_{\Sigma \cap \Sigma'}
	\) if
	\( \Sigma' \supseteq \Phi^+ \) is parabolic and
	\( K \) is local. Take
	\( g \in G^0_\Sigma(K) \cap G^0_{\Sigma'}(K) \). By lemma
	\ref{gauss} we can multiply
	\( g \) from both sides by elements of
	\( G^0_{\Sigma \cap \Sigma'} \) to get a new element
	\(
		g'
		\in
		G_{\Sigma \cap \Phi^-}(K)
		\cap
		G^0_{\Sigma'}(K)
	\). But such
	\( g' \) necessarily lies in
	\( G_{\Sigma \cap \Sigma' \cap \Phi^-} \).

	Now let us prove the second claim. Without loss of generality
	\( \Sigma' \subseteq \Phi^+ \), so
	\[
		G^0_\Sigma \cap G_{\Sigma'}
		=
		G^0_\Sigma \cap G^0_{\Phi^+} \cap G_{\Sigma'}
		=
		G^0_{\Sigma \cap \Phi^+} \cap G_{\Sigma'}
		=
		G_{\Sigma \cap \Phi^+ \cap \Sigma'}
		=
		G_{\Sigma \cap \Sigma'}.
		\qedhere
	\]
\end{proof}

Recall
\cite[\S 4]{twi-for} that a
\textit{%
	\( 2 \)-step
	\( K \)-module%
}
\( (M, M_0) \) consists of
\begin{itemize}

	\item
	a group
	\( M \) with the group operation
	\( \dotplus \);

	\item
	a central subgroup
	\( M_0 \leq M \) such that
	\( [M, M]^\cdot \leq M_0 \);

	\item
	a left
	\( K \)-module structure on
	\( M_0 \);

	\item
	a right action
	\( ({-}) \cdot ({=}) \colon M \times K \to M \) of the multiplicative monoid by group automorphisms such that
	\(
		[m \cdot k, m' \cdot k']^\cdot
		=
		k k' [m, m']^\cdot
	\),
	\(
		m \cdot (k + k')
		=
		m \cdot k
		\dotplus
		k k' \tau(m)
		\dotplus
		m \cdot k'
	\) for some (uniquely determined)
	\( \tau(m) \in M_0 \),
	\( m_0 \cdot k = k^2 m_0 \) for
	\( m_0 \in M_0 \).

\end{itemize}
Then
\begin{align*}
	m \cdot 0 &= \dot 0,
	&
	m \cdot (-k) &= k^2 \tau(m) \dotminus m \cdot k,
	\\
	\tau(\dot 0) &= \dot 0,
	&
	\tau(m \dotplus m')
	&=
	\tau(m) + [m, m']^\cdot + \tau(m'),
	\\
	\tau(\dotminus m) &= \dotminus \tau(m),
	&
	\tau(m \cdot k) &= k^2 \tau(m),
	\\
	\tau(m_0) &= 2 m_0
	\text{ for }
	m_0 \in M_0.
\end{align*}
A pair of subsets
\( (X, X_0) \subseteq (M, M_0) \) (i.e.
\( X \subseteq M \) and
\( X_0 \subseteq M_0 \))
\textit{generates}
\( (M, M_0) \) if
\( X_0 \) generates the
\( K \)-module
\( M_0 \) and
\( X \) generates the
\( K \)-module
\( M / M_0 \).

We say that a
\( 2 \)-step nilpotent
\( K \)-module
\( (M, M_0) \) is
\textit{locally free} if
\( M_0 \) and
\( M / M_0 \) are finitely generated projective
\( K \)-modules. In this case
\( (M, M_0) \)
\textit{splits}, i.e.
\( M = M_0 \dotoplus M_1 \) for some
\( K \)-module
\( M_1 \cong M / M_0 \) with bilinear map
\( c \colon M_1 \times M_1 \to M_0 \) and the operations are given by
\begin{align*}
	(m_0 \dotoplus m_1)
	\dotplus
	(m'_0 \dotoplus m'_1)
	&=
	(m_0 + c(m_1, m'_1) + m'_0)
	\dotoplus
	(m_1 + m'_1),
	\\
	(m_0 \dotoplus m_1) \cdot k
	&=
	k^2 m_0 \dotoplus k m_1,
	\\
	\tau(m_0 \dotoplus m_1)
	&=
	(2 m_0 - c(m_1, m_1)) \dotoplus 0.
\end{align*}

Every locally free
\( 2 \)-step nilpotent
\( K \)-module
\( (M, M_0) \) determines a representable fpqc sheaf
\( \mathbb M_0(E) = E \otimes_K M_0 \),
\(
	\mathbb M(E)
	=
	M \boxtimes E
	=
	E \otimes_K M_0 \dotoplus E \otimes_K M_1
\) of locally free
\( 2 \)-step nilpotent modules (where
\( K \to E \) is a ring homomorphism), the latter one is independent on the choice of the splitting
\cite[\S 4]{twi-for}. Also, locally free
\( 2 \)-step nilpotent
\( K \)-modules satisfy the fpqc descent. If
\( G \) is a group scheme acting on such a sheaf
\( (\mathbb M, \mathbb M_0) \) by automorphisms of
\( 2 \)-step nilpotent modules and
\( G \) stabilizes some generating set
\( (X, X_0) \subseteq (M, M_0) \), then the action is trivial.

By
\cite[\S 2]{loc-iso-ele} there are canonical homomorphisms
\(
	t_\alpha
	\colon
	\mathfrak g_\alpha
	\to
	G_\alpha(K)
\) for non-ultrashort
\( \alpha \in \Phi \) and
\(
	t_\alpha
	\colon
	\mathfrak g_\alpha
	\dotoplus
	\mathfrak g_{2 \alpha}
	\to
	G_\alpha(K)
\) for ultrashort
\( \alpha \). In the second case
\(
	\mathfrak g_\alpha
	\dotoplus
	\mathfrak g_{2 \alpha}
\) means a locally free
\( 2 \)-step nilpotent
\( K \)-module with the canonical short exact sequence
\(
	0
	\to
	\mathfrak g_\alpha
	\to
	\mathfrak g_\alpha
	\dotoplus
	\mathfrak g_{2 \alpha}
	\to
	\mathfrak g_{2 \alpha}
	\to
	0
\), though its splitting is non-canonical. Since
\( G \) is root graded locally in Zariski topology (by our additional assumption on isotropic data), the Lie bracket
\(
	\mathfrak g_\alpha
	\times
	\mathfrak g_\beta
	\to
	\mathfrak g_{\alpha + \beta}
\) is non-degenerate on the second argument if
\( \alpha \) is long and
\( \frac \pi 2 < \angle (\alpha, \beta) < \pi \), i.e.
\(
	\mathfrak g_\beta
	\to
	\mathrm{Hom}(
		\mathfrak g_\alpha,
		\mathfrak g_{\alpha + \beta}
	)
\)
is injective. There are also several other instances of non-degeneracy, we check them case by case.

\begin{lemma}
	\label{c2-nondeg}
	Suppose that
	\( \Phi \) is of type
	\( \mathsf C_2 \). Then for any long
	\( \alpha \) and short
	\( \beta \) such that
	\( \angle(\alpha, \beta) = \frac {3 \pi} 4 \) the Lie bracket
	\(
		\mathfrak g_\alpha
		\times
		\mathfrak g_\beta
		\to
		\mathfrak g_{\alpha + \beta}
	\) is non-degenerate.
\end{lemma}
\begin{proof}
	Without loss of generality,
	\( G \) splits. For the Tits indices
	\( \mathsf B_{n, 2} \) and
	\( \up k {\mathsf D_{n, r}^{(1)}} \) we have
	\( \mathfrak g_\alpha \cong K \) and
	\( \mathfrak g_\beta \cong K^n \), the pairing is isomorphic to
	\(
		K^n \times K \to K^n,\,
		(m, x) \mapsto m x
	\).
	
	For the Tits indices
	\( \mathsf C_{2 d, 2}^{(d)} \) or
	\( \up 1 {\mathsf D_{2 d, 2}^{(d)}} \) we have
	\( \mathfrak g_\beta \cong \mat(d, K) \) and
	\[
		\mathfrak g_\alpha
		\cong
		\Lambda
		=
		\{
			x \in \mat(d, K)
			\mid
			x^{\mathrm t} = \pm x
		\},
	\]
	the pairing is isomorphic to
	\(
		\mat(d, K) \times \Lambda \to \mat(d, K),\,
		(x, u) \mapsto x u
	\). This pairing is also non-degenerate because the sign
	\( - \) appears in
	\( \Lambda \) for the Tits index
	\( \up 1 {\mathsf D_{2 d, 2}^{(d)}} \) and in this case
	\( d \geq 2 \).
	
	Finally, if the Tits index is
	\( \up 2 {\mathsf A_{4 d - 1, 2}^{(d)}} \), then
	\(
		\mathfrak g_\beta
		\cong
		\mat(d, K) \times \mat(d, K)
	\),
	\( \mathfrak g_\alpha \cong \mat(d, K) \), and the pairing is isomorphic to
	\[
		(\mat(d, K) \times \mat(d, K))
		\times
		\mat(d, K)
		\to
		\mat(d, K) \times \mat(d, K),\,
		((x, y), u) \mapsto (x u, - u y).
		\qedhere
	\]
\end{proof}

\begin{lemma}
	\label{bc2-nondeg}
	Suppose that
	\( \Phi \) is of type
	\( \mathsf{BC}_2 \) and
	\( \alpha, \beta \in \Phi \) are orthogonal ultrashort roots. Then the Lie bracket
	\(
		\mathfrak g_\alpha
		\times
		\mathfrak g_\beta
		\to
		\mathfrak g_{\alpha + \beta}
	\) is non-degenerate.
\end{lemma}
\begin{proof}
	In the classical cases fppf locally the pairing is isomorphic to the product map
	\[
		\mat(d, n, K) \times \mat(n, d, K)
		\to
		\mat(d, K)
	\]
	or to a pair of such maps. The exceptional cases
	\( \up 2 {\mathrm E_{6, 2}^{16'}} \),
	\( \mathrm E_{7, 2}^{31} \), and
	\( \mathsf E_{8, 2}^{66} \) can be checked using
	\cite{atlas}.
\end{proof}

\section{Some normalizers and centralizers}

\begin{theorem}
	\label{long-norm}
	Let
	\( G \) be an isotropic reductive group over
	\( K \) with root system
	\( \Phi \) of rank
	\( \geq 2 \). Let also
	\( \alpha \in \Phi \) be a long root and
	\( X \subseteq \mathfrak g_\alpha \) be a generating set of a
	\( K \)-module. Then the group subscheme
	\[
		N
		=
		\{
			g \in G
			\mid
			\forall x \in X \enskip
				\up g {t_\alpha(x)} \in G_\alpha
		\}
	\]
	coincides with the parabolic subgroup
	\(
		P
		=
		G^0_{
			\{
				\beta
				\mid
				\angle(\alpha, \beta) \leq \pi / 2
			\}
		}
	\).
\end{theorem}
\begin{proof}
	Clearly,
	\( P \leq N \) and we may assume that
	\( K \) is local. By lemma
	\ref{gauss} it suffices to check that
	\( G_\Sigma(K) \cap N(K) = 1 \), where
	\(
		\Sigma
		=
		\{
			\beta
			\mid
			\angle(\alpha, \beta) > \frac \pi 2
		\}
	\). Choose a linear map
	\( f \colon \mathbb R \Phi \to \mathbb R \) such that
	\(
		\Ker(f) \cap \Phi
		=
		\mathbb R \alpha \cap \Phi
	\) and let
	\(
		\Phi_s
		=
		\{
			\beta \in \Phi
			\mid
			\mathrm{sign}(f(\beta)) = s
		\}
	\),
	\( \Sigma_s = \Phi_s \cap \Sigma \) for
	\( s \in \{ {-}, 0, {+} \} \), so
	\( \Phi_{\pm} \) are unipotent subsets.

	Take any element
	\( g = g_{-} g_0 g_{+} \in G_\Sigma(K) \cap N(K) \) with
	\( g_s \in G_{\Sigma_s} \), then
	\(
		g_{\pm}
		=
		\prod_{
			\beta
			\in
			\Sigma_{\pm} \setminus 2 \Sigma_{\pm}
		}
			t_\beta(y_\beta)
	\) for any linear orders on
	\( \Sigma_{\pm} \setminus 2 \Sigma_{\pm} \) and
	\( g_0 \in U_{-\alpha}(K) \) (or
	\( g_0 \in U_{-\alpha / 2}(K) \) if
	\( \Phi \) is of type
	\( \mathsf{BC}_\ell \)). For any
	\( x \in X \) there is
	\( x' \in \mathfrak g_\alpha \) such that
	\( g\, t_\alpha(x) = t_\alpha(x')\, g \). Comparing both sides of
	\(
		g_{-}\,
		g_0\,
		t_\alpha(x)\,
		g_{+}^{t_\alpha(x)}
		=
		\up{t_\alpha(x')}{g_{-}}\,
		t_\alpha(x')\,
		g_0\,
		g_{+}
	\) we get
	\( [g_{+}, t_\alpha(x)] = 1 \) and
	\( [g_{-}, t_\alpha(x')] = 1 \). Since
	\( x \) runs over a generating set of
	\( \mathfrak g_\alpha \) and the commutator map
	\(
		\mathfrak g_\alpha
		\times
		\mathfrak g_\beta
		\to
		\mathfrak g_{\alpha + \beta}
	\) is non-degenerate on the second argument for
	\( \beta \in \Sigma_{\pm} \),
	\( g_{+} \) is trivial, so
	\( g_{-} \) is also trivial by changing the sign of
	\( f \).

	It follows that
	\( g = g_0 \in U_{-\alpha'}(K) \) for
	\( \alpha' \in \{ \alpha, \frac \alpha 2 \} \). Now suppose that
	\( \Phi \) is not of the type
	\( \mathsf{BC}_\ell \) and take a root
	\( \beta \in \Phi \) such that
	\( \frac \pi 2 < \angle(\alpha, \beta) < \pi \) and
	\( \beta \) is long if
	\( \Phi \) is of the type
	\( \mathsf G_2 \). Conjugate
	\( t_\beta(z) \) by both sides of
	\(
		g\, t_\alpha(x)
		=
		t_\alpha(x')\, g
	\). We get
	\( \bigl[ \log g, [x, z] \bigr] = 0 \), where
	\( \log g \in \mathfrak g_{- \alpha} \) is the element such that
	\( g = t_{- \alpha}(\log g) \). Since
	\(
		\mathfrak g_{\alpha + \beta}
		=
		[\mathfrak g_\alpha, \mathfrak g_\beta]
	\),
	\( [\log g, \mathfrak g_{\alpha + \beta}] = 0 \) and
	\( \log g = 0 \) (by lemma
	\ref{c2-nondeg} if
	\( \beta \) is short), i.e.
	\( g = 1 \).

	Finally, suppose that
	\( \Phi \) is of type
	\( \mathsf{BC}_\ell \). Take an ultrashort root
	\( \beta \) orthogonal to
	\( \alpha \) and conjugate
	\( t_\beta(u) \) by both sides of
	\( t_\alpha(x)\, g^{-1} = g^{-1}\, t_\alpha(x') \). We get
	\( [t_\alpha(x'), [g^{-1}, t_\beta(u)]] = 1 \), so
	\( [g^{-1}, t_\beta(u)] = 1 \) and
	\( g \in U_{-\alpha}(K) \) by lemma
	\ref{bc2-nondeg}. In other words, we reduce to the already considered case of
	\( \mathsf C_\ell \subseteq \mathsf{BC}_\ell \).
\end{proof}

We need another example of locally free
\( 2 \)-step nilpotent
\( K \)-modules. Let
\( G \) be an isotropic reductive group scheme over
\( K \) with root system
\( \Phi \) and
\( f \colon \mathbb R \Phi \to \mathbb R \) be a linear map such that
\( f(\Phi) \subseteq \{ -2, -1, 0, 1, 2 \} \), i.e. a ``%
\( 5 \)-grading''. Then
\[
	\bigl(
		G_{\Phi \cap f^{-1}(\{ 1, 2 \})},
		G_{\Phi \cap f^{-1}(2)}
	\bigr)
\]
is a sheaf of locally free
\( 2 \)-step nilpotent modules. The action of the group scheme
\( G^0_{\Phi \cap \Ker(f)} \) on it commutes with the operations of
\( 2 \)-step nilpotent modules, this can be easily checked by passing to a split form of
\( G \).

\begin{lemma}
	\label{urad-cent}
	Let
	\( G \) be an isotropic reductive group scheme and
	\( f \colon \mathbb R \Phi \to \mathbb R \) a linear map such that
	\( f(\Phi) \subseteq \{ -2, -1, 0, 1, 2 \} \) and
	\( \Phi \cap f^{-1}(2) \neq \varnothing \). Then
	\[
		C
		=
		\bigl\{
			g \in G^0_{\Phi \cap \Ker(f)}
			\mid
			[g, t_\alpha(x_{\alpha, i})] = 1
			\text{ for all }
			\alpha
			\text{ with }
			f(\alpha) > 0
		\bigr\}
	\]
	coincides with the scheme center
	\( \mathrm C(G) \leq \mathrm C(L) \),
	\( x_{\alpha, i} \in \mathfrak g_\alpha \) are
	\( K \)-module generators if
	\( \alpha \) is not ultrashort and
	\(
		x_{\alpha, i}
		\in
		\mathfrak g_\alpha
		\dotoplus
		\mathfrak g_{2 \alpha}
	\) generate the
	\( K \)-module
	\( \mathfrak g_\alpha \) otherwise.
\end{lemma}
\begin{proof}
	Clearly,
	\( C \) is the scheme centralizer of
	\( G_{\Phi_{> 0}} \) in
	\( G^0_{\Phi \cap \Ker(f)} \), where
	\(
		\Phi_{> 0}
		=
		\{ \alpha \in \Phi \mid f(\alpha) > 0 \}
	\). We can assume that
	\( G \) splits and
	\( \Phi = \widetilde \Phi \). Firstly, let us show that
	\( C \) is contained in the maximal torus
	\( T = L \). By theorem
	\ref{long-norm} and lemma
	\ref{subgr-int} the group scheme
	\( C \) is contained in
	\( G^0_\Sigma \), where
	\( \Sigma \) consists of the roots
	\( \alpha \in \Phi_0 \) such that
	\( \angle(\alpha, \beta) \leq \frac \pi 2 \) for all long roots
	\( \beta \in \Phi_{> 0} \). By considering rank
	\( 2 \) root subsystems it follows that roots of
	\( \Sigma \) are actually orthogonal to long roots from
	\( \Phi_{> 0} \).

	Now if the set of long roots of
	\( \Phi \) is an indecomposable root system of full rank, then the set of long roots from
	\( \Phi \setminus \Ker(f) \) generates this root subsystem e.g. by
	\cite[lemma 1]{loc-iso-ele}. In this case necessarily
	\( \Sigma = \varnothing \). Otherwise
	\( \Phi = \mathsf C_\ell \), possibly
	\( \mathsf B_2 = \mathsf C_2 \). Up to the choice of the base
	\( \Delta \subseteq \Phi \) any
	\( 5 \)-grading of
	\( \mathsf C_\ell \) maps basic roots to
	\( 0 \) except one, and this distinguished root is mapped to
	\( 1 \) (if it is short) or
	\( 2 \) (if it is long). It easily follows that
	\(
		\Sigma
		=
		\mathsf C_m
		\subseteq
		\mathsf C_\ell
	\) is a proper root subsystem generated by a connected Dynkin subdiagram containing the long basic root or it is empty. But the representation of
	\( \mathrm{Sp}(2 m, K) \) on
	\(
		G_{
			\{
				\alpha \in \mathsf C_{m + 1}
				\mid
				\alpha \cdot \mu > 0
			\}
		}
		/
		U_\mu
	\) is the defining representation of symplectic group scheme, where
	\( \mu \in \mathsf C_{m + 1} \) is the highest root. It follows that
	\( C \leq T \) in all cases.

	It remains to show that
	\( C \) trivially acts on all
	\( G_\alpha \). But all roots from
	\( \Phi_{> 0} \) are trivial on
	\( C \) (recall that roots are some homomorphisms
	\( T \to \mathbb G_{\mathrm m} \)) and such roots generate the root lattice
	\( \mathbb Z \Phi \).
\end{proof}

If
\( \Phi \) is of type
\( \mathsf{BC}_\ell \) with
\( \ell \geq 2 \), let
\(
	\mathrm C^{\mathrm{us}}(G)
	=
	\bigcap_{\alpha \in \mathsf C_\ell}
		\Cent_L(U_\alpha)
\) be the scheme centralizer of non-ultrashort root subgroups. This group scheme clearly contains the center and it is contained in
\( L \) by theorem
\ref{long-norm}.  It turns out that it is reductive or finite of multiplicative type, the second case holds only for
\( \mathsf E_{8, 2}^{66} \). The derived subgroup of
\( \mathrm C^{\mathrm{us}}(G) \) corresponds to the ``middle'' component of
\( u^{-1}(0) \) in the classical cases and to the component
\( \mathsf A_1 \) in the case
\( \mathsf E_{7, 2}^{31} \). More precisely, assuming that
\( G \) is split and adjoint connected, we have the following.

\begin{center}
	\begin{tabular}{|c|c|c|}
		\hline
		Tits index
		& \( \Cent^{\mathrm{us}}(G) \)
		& conditions
		\\ \hline \hline
		\( \up 2 {\mathsf A_{n, r}^{(d)}} \)
		& \( \mathbb{GL}_{n + 1 - 2 r d} \)
		& \( d \mid n + 1 \),
		\( 2 r d \leq n \),
		\( d \geq 1 \),
		\( r \geq 2 \)
		\\ \hline
		\( \mathsf C_{n, r}^{(d)} \)
		& \( \mathbb{S}\mathrm{p}_{2 n - 2 r d} \)
		& \( 2 \leq d = 2^k \mid 2 n \),
		\( r d \leq n - 1 \),
		\( r \geq 2 \)
		\\ \hline
		\( \up 1 {\mathsf D_{n, r}^{(d)}} \),
		\( \up 2 {\mathsf D_{n, r}^{(d)}} \)
		& \( \mathbb{SO}_{2 n - 2 r d} \)
		& \( r \geq 2 \),
		\( 2 \leq d = 2^k \mid 2 n \),
		\( r d \leq n - 1 \)
		\\ \hline
		\( \up 2 {\mathsf E_{6, 2}^{16'}} \)
		& \( \mathbb G_{\mathrm m} \)
		&
		\\ \hline
		\( \mathsf E_{7, 2}^{31} \)
		& \( \mathbb{SL}_2 \)
		&
		\\ \hline
		\( \mathsf E_{8, 2}^{66} \)
		& \( \mu_2 \)
		&
		\\ \hline
	\end{tabular}
\end{center}

In the classical cases the group scheme
\( \Cent^{\mathrm{us}}(G) \) can be computed using the block structure on a finite central extension of
\( G \) as a matrix group. In the exceptional cases we apply theorem
\ref{long-norm} to the splitting isotropic structure on
\( G \) and get
\( \Cent^{\mathrm{us}}(G) \leq T \) (for
\( \up 2 {\mathsf E_{6, 2}^{16'}} \) and
\( \mathsf E_{8, 2}^{66} \)) or
\(
	\mathbb{SL}_2
	\leq
	\Cent^{\mathrm{us}}(G)
	\leq
	T\, \mathbb{SL}_2
\), where
\( T \) is the maximal torus from the splitting isotropic structure. The intersection
\( T \cap \Cent^{\mathrm{us}}(G) \) is easy to compute using the root diagrams
\cite{atlas}.

For any root
\( \alpha \in \Phi \) let
\(
	\Gamma_\alpha
	=
	\{
		\beta \in \Phi
		\mid
		\alpha + \beta \notin \Phi \cup \{ 0 \}
	\}
\) and
\(
	Z_\alpha
	=
	\bigcup_{\beta \in \Gamma_\alpha}
		t_\beta(X_\beta)
	\subseteq
	G(K)
\) for some generating sets
\( X_\beta \subseteq \mathfrak g_\beta \) for non-ultrashort
\( \beta \) and generating sets
\(
	(X_\beta, X_{2 \beta})
	\subseteq
	(
		\mathfrak g_\beta
		\dotoplus
		\mathfrak g_{2 \beta},
		\mathfrak g_{2 \beta}
	)
\) for ultrashort
\( \beta \). It is easy to see by considering all rank
\( 2 \) and rank
\( 3 \) irreducible root systems that
\( \Gamma_\alpha \) is closed and
\( \angle(\alpha, \beta) \leq \frac \pi 2 \) for all
\( \beta \in \Gamma_\alpha \).

\begin{theorem}
	\label{dbl-centzer}
	Let
	\( K \) be a commutative ring,
	\( G \) an isotropic reductive group scheme over
	\( K \) with root system
	\( \Phi \) of rank at least
	\( 2 \), and
	\( \alpha \in \Phi \) a root.
	\begin{itemize}

		\item
		If the type of
		\( \Phi \) is
		\( \mathsf{BC}_\ell \) and
		\( \alpha \) is ultrashort, then
		\(
			\Cent_G(Z_\alpha)
			=
			\Cent^{\mathrm{us}}(G)\, U_\alpha
		\).

		\item
		If
		\( \alpha = \mathrm e_i + \mathrm e_j \) (using the convention
		\( \mathrm e_{-i} = -\mathrm e_i \)) is short and the type of
		\( \Phi \) is
		\( \mathsf C_\ell \) or
		\( \mathsf{BC}_\ell \), including
		\( \mathsf B_2 = \mathsf C_2 \), then
		\(
			\Cent_G(Z_\alpha)
			=
			\Cent(G)\,
			U_{2 \mathrm e_i}\,
			U_{\mathrm e_i + \mathrm e_j}\,
			U_{2 \mathrm e_j}
		\).

		\item
		Otherwise
		\( \Cent_G(Z_\alpha) = \Cent(G)\, U_\alpha \).

	\end{itemize}
\end{theorem}
\begin{proof}
	By theorem
	\ref{long-norm} and lemma
	\ref{subgr-int}
	\( \mathrm C_G(Z_\alpha) \leq G^0_\Sigma \), where
	\[
		\textstyle \Sigma
		=
		\{
			\gamma \in \Phi
			\mid
			\angle(\beta, \gamma) \leq \frac \pi 2
			\text{ for all long }
			\beta \in \Gamma_\alpha
		\}.
	\]
	Considering saturated root subsystems of rank
	\( 2 \) and
	\( 3 \) containing
	\( \alpha \) and arbitrary
	\( \beta \in \Sigma \) (or just explicitly calculating
	\( \Gamma_\alpha \) in all cases) we see that
	\begin{itemize}

		\item
		If
		\( \Phi \) is of type
		\( \mathsf C_\ell \) with
		\( \ell \geq 2 \) and
		\( \alpha = \mathrm e_i + \mathrm e_j \) is short (recall that
		\( \mathrm e_{-i} = -\mathrm e_i \)), then
		\(
			\Sigma
			=
			\{
				2 \mathrm e_i,
				\mathrm e_i + \mathrm e_j,
				2 \mathrm e_j
			\}
		\).

		\item
		If
		\( \Phi \) is of type
		\( \mathsf{BC}_\ell \) with
		\( \ell \geq 2 \) and
		\( \alpha \) is long, then
		\( \Sigma = \{ \alpha, \frac 1 2 \alpha \} \).

		\item
		If
		\( \Phi \) is of type
		\( \mathsf{BC}_\ell \) with
		\( \ell \geq 2 \) and
		\( \alpha = \mathrm e_i + \mathrm e_j \) is short, then
		\(
			\Sigma
			=
			\{
				\mathrm e_i,
				2 \mathrm e_i,
				\mathrm e_i + \mathrm e_j,
				\mathrm e_j,
				2 \mathrm e_j
			\}
		\).

		\item
		If
		\( \Phi \) is of type
		\( \mathsf{BC}_\ell \) with
		\( \ell \geq 2 \) and
		\( \alpha \) is ultrashort, then
		\( \Sigma = \{ \alpha, 2 \alpha \} \).

		\item
		Otherwise
		\( \Sigma = \{ \alpha \} \).

	\end{itemize}
	
	Now apply lemma
	\ref{bc2-nondeg} to get
	\[
		\Cent_G(Z_\alpha)
		\leq
		\begin{cases}
			G^0_{\{ \alpha, 2 \alpha \}}
			&
			\text{%
				if
				\( \Phi \) is of type
				\( \mathsf{BC}_\ell \),
				\( \ell \geq 2 \),
				\( \alpha \) is ultrashort%
			};
			\\
			G^0_{
				\{
					2 \mathrm e_i,
					\mathrm e_i + \mathrm e_j,
					2 \mathrm e_j
				\}
			}
			& \text{%
				if
				\( \Phi \) is of type
				\( \mathsf C_\ell \) or
				\( \mathsf{BC}_\ell \),
				\( \ell \geq 2 \),
				\(
					\alpha
					=
					\mathrm e_i + \mathrm e_j
				\) is short%
			};
			\\
			G^0_{\{ \alpha \}}
			&
			\text{otherwise}.
		\end{cases}
	\]

	It remains to show that
	\( \Cent_L(Z_\alpha) = \mathrm C(G) \) or
	\( \mathrm C^{\mathrm{us}}(G) \). If
	\( \alpha \) is neither ultrashort nor short in
	\( \Phi = \mathsf G_2 \), we are done by lemma
	\ref{urad-cent} applied to
	\(
		f(\beta)
		=
		2 \frac{\alpha \cdot \beta}{\alpha \cdot \alpha}
	\). In the case of
	\( \mathsf G_2 \) first apply the same argument to
	\( G_{\{ - \alpha, \alpha \}} \) to get
	\( \Cent_L(Z_\alpha) \leq \Cent(L) \) (using
	\( [L, L] \leq G_{\{ - \alpha, \alpha \}} \)), and then note that
	\( \Cent_{\Cent(L)}(Z_\alpha) = \Cent(G) \) because the preimage of
	\( \Gamma_\alpha \) in
	\( \widetilde \Phi \) spans the whole root lattice e.g. by considering root diagrams
	\cite{atlas}.
	
	It remains to consider
	\( \Phi = \mathsf{BC}_\ell \) and ultrashort
	\( \alpha \). Note that
	\( \Gamma_\alpha \) is a parabolic subset of
	\( \mathsf C_\ell \subseteq \mathsf{BC}_\ell \). Applying lemma
	\ref{urad-cent} to
	\( G_{\mathsf C_\ell} \) with
	\(
		f(\beta)
		=
		2 \frac{\beta \cdot \alpha}{\alpha \cdot \alpha}
	\) we see that
	\( \Cent_L(Z_\alpha) \) centralizes all
	\( U_\beta \) with
	\( \beta \in \mathsf C_\ell \) (note that
	\( G_{\mathsf C_\ell} \leq G^0_{\mathsf C_\ell} \) is a simple factor e.g. by
	\cite[lemma 1]{loc-iso-ele} and all other factors centralize
	\( G_{\mathsf C_\ell} \supset Z_\alpha \)), so
	\(
		\Cent_L(Z_\alpha)
		=
		\Cent^{\mathrm{us}}(G)
	\).
\end{proof}

\begin{theorem}
	\label{cent-norm}
	Let
	\( G \) an isotropic reductive group scheme over commutative ring
	\( K \) with root system
	\( \Phi \) of rank at least
	\( 2 \). Let also
	\( X_\alpha \subseteq \mathfrak g_\alpha \) be generating subsets for non-ultrashort
	\( \beta \) and generating sets
	\(
		(X_\alpha, X_{2 \alpha})
		\subseteq
		(
			\mathfrak g_\alpha
			\dotoplus
			\mathfrak g_{2 \alpha},
			\mathfrak g_{2 \alpha}
		)
	\) otherwise. Then
	\[
		\Cent_G\bigl(
			\bigcup_{\alpha \in \Phi}
				t_\alpha(X_\alpha)
		\bigr)
		=
		\Cent(G),
		\quad
		\{
			g \in G
			\mid
			\forall \alpha \in \Phi \enskip
				\up g {t_\alpha(X_\alpha)}
				\subseteq
				U_\alpha
		\}
		= L.
	\]
\end{theorem}
\begin{proof}
	The first claim follows from theorem
	\ref{dbl-centzer} applied to opposite long roots and lemma
	\ref{subgr-int}. The second one is a corollary of theorem
	\ref{long-norm} applied to all long roots and lemma
	\ref{subgr-int}.
\end{proof}

\section{Main results}

Recall that
\(
	\mathrm E_G(K)
	=
	\bigl\langle
		U_\alpha(K)
		\mid
		\alpha \in \Phi
	\bigr\rangle
\) is the
\textit{elementary subgroup}
\cite{iso-ele-nor, loc-iso-ele}. If
\( \mathrm E_G(K) \leq H \leq G(K) \) is an intermediate subgroup, then
\(
	\Cent(G(K)) \cap H
	=
	\Cent(H)
	=
	\Cent_H(\mathrm E_G(K))
\) by theorem
\ref{cent-norm}.

\begin{theorem}
	\label{diophantine}
	Let
	\( G \) be an isotropic reductive group scheme over
	\( K \) such that the rank of
	\( \Phi \) is at least
	\( 2 \). Then the sets
	\( L(K) \),
	\( \Cent(G)(K) \), and all
	\( \Cent(G)(K)\, U_\alpha(K) \) are Diophantine in the group
	\( G(K) \). Moreover, the root subgroups
	\( U_\alpha(K) \) are Diophantine in the group
	\( G(K) \) unless the type is
	\( \mathsf B_2 = \mathsf C_2 \),
	\( \mathsf{BC}_2 \), or
	\( \alpha \) is short and the type is
	\( \mathsf G_2 \). If
	\( \mathrm E_G(K) \leq H \leq G(K) \) is an intermediate subgroup, then the intersections of all above subgroups with
	\( H \) are Diophantine in
	\( H \).
\end{theorem}
\begin{proof}
	By theorem
	\ref{dbl-centzer} the following subgroups are Diophantine.
	\begin{itemize}

		\item
		\( \Cent^{\mathrm{us}}(G)(K)\, U_\alpha(K) \) for ultrashort
		\( \alpha \) and
		\( \Phi \) of type
		\( \mathsf{BC}_\ell \);

		\item
		\(
			\Cent(G)(K)\,
			U_{2 \mathrm e_i}(K)\,
			U_{\mathrm e_i + \mathrm e_j}(K)\,
			U_{2 \mathrm e_j}(K)
		\) for
		\( \alpha = \mathrm e_i + \mathrm e_j \) and
		\( \Phi \) of type
		\( \mathsf C_\ell \) or
		\( \mathsf{BC}_\ell \),
		\( \ell \geq 2 \);

		\item
		\( \mathrm C(G)(K)\, U_\alpha(K) \) for
		\( \alpha \) long or
		\( \alpha \) short and
		\( \Phi \) of type
		\( \mathsf B_\ell \),
		\( \ell \geq 3 \),
		\( \mathsf F_4 \), or
		\( \mathsf G_2 \).

	\end{itemize}
	Intersections of Diophantine subgroups, their products and their commutators with single elements are also Diophantine.

	Clearly,
	\[
		\Cent(G)(K)
		=
		\Cent(G)(K)\, U_\alpha(K)
		\cap
		\Cent(G)(K)\, U_{-\alpha}(K)
	\]
	is Diophantine, where
	\( \alpha \) is any long root. By theorem
	\ref{long-norm} applied to
	\( G / \Cent(G) \) we have
	\[
		\{
			g \in G(K)
			\mid
			\up g {t_\alpha(X)}
			\subseteq
			\Cent(G)(K)\, U_\alpha(K)
		\}
		=
		G^0_{
			\{
				\beta
				\mid
				\angle(\alpha, \beta) \leq \pi / 2
			\}
		}(K)
	\]
	for every long root
	\( \alpha \), where
	\( X \subseteq \mathfrak g_\alpha \) is a finite generating set. The intersection of all these parabolic subgroups is
	\( L(K) \) by lemma
	\ref{subgr-int}.

	If
	\( \alpha \) is long and the type of
	\( \Phi \) is neither
	\( \mathsf C_\ell \) nor
	\( \mathsf{BC}_\ell \), then
	\[
		U_\alpha(K)
		=
		\prod_{x \in X}
			\bigl[
				\Cent(G)(K)\, U_\beta(K),
				t_{\alpha - \beta}(x)
			\bigr]
	\]
	is Diophantine, where
	\( \beta \) is long,
	\( \angle(\alpha, \beta) = \frac \pi 3 \), and
	\( X \subseteq \mathfrak g_{\alpha - \beta} \) is a finite generating set.

	If
	\( \alpha = \mathrm e_i \) is short and the type of
	\( \Phi \) is
	\( \mathsf B_\ell \),
	\( \ell \geq 3 \), then
	\[
		U_\alpha(K)
		=
		\Cent(G)(K)\, U_\alpha(K)
		\cap
		G_{\mathrm e_i + \mathrm e_j}(K)\,
		\prod_{x \in X}
			\bigl[
				\Cent(G)(K)\, G_{\mathrm e_j}(K),
				t_{\mathrm e_i - \mathrm e_j}(x)
			\bigr]
	\]
	is Diophantine, where
	\( j \neq \pm i \) is another index and
	\(
		X
		\subseteq
		\mathfrak g_{\mathrm e_i - \mathrm e_j}
	\) is a finite generating set. It follows that all root subgroups are Diophantine for the type
	\( \mathsf F_4 \).

	Now suppose that
	\( \Phi \) is of type
	\( \mathsf C_\ell \) or
	\( \mathsf{BC}_\ell \) and
	\( \ell \geq 3 \). We have
	\[
		U_{\mathrm e_i + \mathrm e_j}(K)
		=
		\prod_{x \in X,\, x' \in X'}
			\bigl[
				\bigl[
					\Cent(G)(K)\,
					U_{2 \mathrm e_j}(K)\,
					U_{\mathrm e_j + \mathrm e_k}(K)\,
					U_{2 \mathrm e_k}(K),
					t_{\mathrm e_i - \mathrm e_j}(x)
				\bigr],
				t_{\mathrm e_j - \mathrm e_k}(x')
			\bigr],
	\]
	where
	\( k \notin \{ - i, i, - j, j \} \) is a new index and
	\(
		X
		\subseteq
		\mathfrak g_{\mathrm e_i - \mathrm e_j}
	\),
	\(
		X'
		\subseteq
		\mathfrak g_{\mathrm e_j - \mathrm e_k}
	\) are finite generating sets. Next, the group
	\[
		U_{2 \mathrm e_i}(K)
		=
		\Cent(G)\, U_{2 \mathrm e_i}(K)
		\cap
		U_{\mathrm e_i + \mathrm e_j}(K)\,
		\prod_{x \in X}
			\bigl[
				\Cent(G)\, U_{2 \mathrm e_j}(K),
				t_{\mathrm e_i - \mathrm e_j}(x)
			]
	\]
	is Diophantine, where
	\(
		X
		\subseteq
		\mathfrak g_{\mathrm e_i - \mathrm e_j}
	\) is a finite generating set. If
	\( \Phi \) is of type
	\( \mathsf{BC}_\ell \), then
	\[
		U_{\mathrm e_i}(K)
		=
		\Cent^{\mathrm{us}}(G)(K)\,
		U_{\mathrm e_i}(K)
		\cap
		U_{\mathrm e_i - \mathrm e_j}(K)\,
		U_{\mathrm e_i + \mathrm e_j}(K)\,
		\prod_{x \in X}
			\bigl[
				\Cent^{\mathrm{us}}(G)(K)\,
				U_{\mathrm e_j}(K),
				t_{\mathrm e_i - \mathrm e_j}(x)
			\bigr]
	\]
	is also Diophantine, where
	\(
		X
		\subseteq
		\mathfrak g_{\mathrm e_i - \mathrm e_j}
	\) is a finite generating set.

	Finally, consider the exceptional cases
	\( \mathsf C_2 \) and
	\( \mathsf{BC}_2 \). The group
	\( \Cent(G)(K)\, U_{\mathrm e_i}(K) \) is Diophantine (in the case of
	\( \mathsf{BC}_2 \)) because it contains
	\(
		g
		\in
		\Cent^{\mathrm{us}}(K)\, U_{\mathrm e_i}(K)
	\) if and only if
	\[
		\bigl[ g, t_{\mathrm e_i}(x) \bigr]
		\in
		\Cent(G)(K)\, U_{2 \mathrm e_i}(K)
	\]
	for all
	\( x \in X \), where
	\(
		X
		\subseteq
		\mathfrak g_{\mathrm e_i}
		\dotoplus
		\mathfrak g_{2 \mathrm e_i}
	\) is a finite generating set. We have
	\[
		\Cent(G)\,
		U_{2 \mathrm e_i}(K)\,
		U_{\mathrm e_i + \mathrm e_j}(K)
		=
		\Cent(G)\,
		U_{2 \mathrm e_i}(K)\,
		\prod_{x \in X}
			\bigl[
				\Cent(G)\,
				U_{2 \mathrm e_i}(K)\,
				U_{\mathrm e_i + \mathrm e_j}(K)\,
				U_{2 \mathrm e_j}(K),
				t_{\mathrm e_i - \mathrm e_j}(x)
			\bigr]
	\]
	for a finite generating set
	\(
		X
		\subseteq
		\mathfrak g_{\mathrm e_i - \mathrm e_j}
	\) by lemma
	\ref{c2-nondeg}, so
	\( \Cent(G)\, U_{\mathrm e_i + \mathrm e_j}(K) \) is also Diophantine.
	
	The last claim about intersections with
	\( \mathrm E_G(K) \leq H \leq G(K) \) easily follows from explicit Diophantine formulas for
	\( L(K) \),
	\( \Cent(G)(K) \),
	\( \Cent(G)(K)\, U_\alpha(K) \), and
	\( U_\alpha(K) \).
\end{proof}

Since
\( G \) is a finitely presented affine scheme, the point group
\( G(K) \) is e-interpretable in
\( K \). The following theorem shows the converse.

\begin{theorem}
	\label{e-interpret}
	Let
	\( G \) be an isotropic reductive group scheme over
	\( K \) with the root system
	\( \Phi \) of rank at least
	\( 2 \). Then the ring
	\( K \) is e-interpretable in any intermediate subgroup
	\( \mathrm E_G(K) \leq H \leq G(K) \).
\end{theorem}
\begin{proof}
	We assume that the type of
	\( \Phi \) is neither
	\( \mathsf G_2 \) nor
	\( \mathsf{BC}_\ell \) because these cases easily reduce to
	\( \mathsf A_2 \) and
	\( \mathsf C_\ell \) respectively. Fix some finite
	\( K \)-module generating sets
	\( X_\alpha \subseteq \mathfrak g_\alpha \) for all roots
	\( \alpha \) and let
	\(
		Y_\alpha
		=
		X_\alpha
		\cup
		\bigcup_{\beta, \alpha - \beta \in \Phi}
		[X_\beta, X_{\alpha - \beta}]
	\).
	
	Below we use that up to canonical isomorphisms the groups
	\( \mathfrak g_\alpha \) are e-interpretable in
	\( H \) by theorem
	\ref{diophantine} as Diophantine subgroups
	\( U_\alpha(K) \leq H \) or as factor-groups
	\( \Cent(H)\, U_\alpha(K) / \Cent(H) \). Under these e-interpretations all commutator maps
	\(
		\mathfrak g_\alpha \times \mathfrak g_\beta
		\to
		\mathfrak g_{\alpha + \beta}
	\) with
	\( \alpha, \beta, \alpha + \beta \in \Phi \) have Diophantine graphs.

	Let
	\( \widetilde K \) be the set of families
	\[
		(
			k_\alpha
			\colon
			Y_\alpha
			\to
			\mathfrak g_\alpha(K)
		)_{\alpha \in \Phi}
	\] such that
	\[
		[k_\alpha(y_\alpha), y_\beta]
		=
		[y_\alpha, k_\beta(y_\beta)]
		\in
		\mathfrak g_{\alpha + \beta}
	\]
	for
	\( y_\alpha \in Y_\alpha \),
	\( y_\beta \in Y_\beta \) and
	\[
		[k_\alpha(x_\alpha), x_\beta]
		=
		k_{\alpha + \beta}([x_\alpha, x_\beta])
		\in
		\mathfrak g_{\alpha + \beta}
	\]
	for
	\( x_\alpha \in X_\alpha \),
	\( x_\beta \in X_\beta \). Here
	\( \alpha \) and
	\( \beta \) are all roots such that
	\( \alpha + \beta \) is also a root. Clearly,
	\( \widetilde K \) is e-interpretable in
	\( H \) as a set.

	Every
	\( k \in K \) determines a corresponding element of
	\( \widetilde K \) given by
	\( y_\alpha \mapsto k y_\alpha \). Conversely, take
	\( (k_\alpha)_{\alpha \in \Phi} \in \widetilde K \). For every root
	\( \alpha \) there is long
	\( \beta \in \Phi \) such that
	\( \alpha + \beta \in \Phi \), so
	\(
		[k_\alpha(y_\alpha), y_\beta]
		=
		[y_\alpha, k_\beta(y_\beta)]
	\). Recall that the Lie bracket
	\(
		[{-}, {=}]
		\colon
		\mathfrak g_\alpha \times \mathfrak g_\beta
		\to
		\mathfrak g_{\alpha + \beta}
	\) is non-degenerate (by lemma
	\ref{c2-nondeg} if
	\( \alpha \) is short). For any linear relation
	\(
		\sum_{y_\alpha \in Y_\alpha}
			y_\alpha a_\alpha
		=
		0
	\) the equation implies that
	\(
		\sum_{y_\alpha \in Y_\alpha}
			k_\alpha(y_\alpha) a_\alpha
		=
		0
	\), so
	\( k_\alpha \) continues to a unique linear map
	\(
		k_\alpha
		\colon
		\mathfrak g_\alpha
		\to
		\mathfrak g_\alpha
	\).

	The equations on
	\( (k_\alpha)_\alpha \) imply that
	\(
		[k_\alpha(x), y]
		=
		k_{\alpha + \beta}([x, y])
		=
		[x, k_\beta(y)]
	\). We claim that such maps
	\( k_\alpha \) are necessarily scalar. It suffices to consider
	\( \Phi \) of rank
	\( 2 \) and to check the claim for only one root
	\( \alpha \) (by lemma
	\ref{c2-nondeg}). This is clear if one of root subspaces has rank
	\( 1 \). To prove the claim we can also assume that
	\( G \) splits and choose convenient parametrizations of the root subspaces
	\( \mathfrak g_\alpha \).

	\begin{itemize}
		
		\item
		In the case
		\( \up 1 {\mathsf A_{3 d - 1, 2}^{(d)}} \) we have
		\[
			k_{\mathrm e_1 - \mathrm e_2}(x)\, y
			=
			k_{\mathrm e_1 - \mathrm e_3}(x y)
			=
			x\, k_{\mathrm e_2 - \mathrm e_3}(y)
		\]
		for
		\( x, y \in \mat(d, K) \). Hence
		\(
			k_{\mathrm e_1 - \mathrm e_2}(x)
			=
			k_{\mathrm e_1 - \mathrm e_3}(x)
			=
			k_{\mathrm e_2 - \mathrm e_3}(x)
		\) (i.e. all maps
		\( k_\alpha \) coincide) and
		\(
			k_\alpha(x)
			=
			x\, k_\alpha(1)
			=
			k_\alpha(1)\, x
		\), so
		\( k_\alpha(1) \) lies in the center of
		\( \mat(d, K) \), and this center is precisely
		\( K \).
		
		\item
		The case
		\( \up 1 {\mathsf E_{6, 2}^{28}} \) is similar to the previous one. We have
		\[
			k_{\mathrm e_1 - \mathrm e_2}(x)\, y
			=
			k_{\mathrm e_1 - \mathrm e_3}(x y)
			=
			x\, k_{\mathrm e_2 - \mathrm e_3}(y)
		\]
		for
		\( x, y \in \mathrm Z(K) \), where
		\( \mathrm Z(K) \) is the split octonion algebra over
		\( K \). Hence all maps
		\( k_\alpha \) coincide and
		\(
			k_\alpha(x)
			=
			x\, k_\alpha(1)
			=
			k_\alpha(1)\, x
		\), so
		\( k_\alpha(1) \in K \).
		
		\item
		In the case
		\( \up 2 {\mathsf A_{4 d - 1, 2}^{(d)}} \) we have
		\[
			k_{\mathrm e_1 - \mathrm e_2}(x, y) \circ z
			=
			(x, y) \circ k_{2 \mathrm e_2}(z)
			=
			k_{\mathrm e_1 + \mathrm e_2}((x, y) \circ z)
		\]
		for
		\( x, y, z \in \mat(d, K) \), where
		\( (x, y) \circ z = (z x, y z) \). It follows that
		\(
			k_{\mathrm e_1 + \mathrm e_2}(x, y)
			=
			(x, y) \circ k_{\mathrm e_2}(z)
		\) and
		\(
			k_{\mathrm e_2}(z)
			=
			z k_{\mathrm e_2}(1)
			=
			k_{\mathrm e_2}(1) z
		\), so
		\( k_{\mathrm e_2}(1) \in K \).

		\item
		In the case
		\( \mathsf C_{2 d, 2}^{(d)} \) we have
		\[
			k_{\mathrm e_1 - \mathrm e_2}(x)\, y
			=
			x\, k_{2 \mathrm e_2}(y)
			=
			k_{\mathrm e_1 + \mathrm e_2}(x y)
		\]
		for
		\( x \in \mat(d, K) \) and symmetric matrix
		\( y = y^{\mathrm t} \in \mat(d, K) \). Then
		\(
			k_{2 \mathrm e_2}(y)
			=
			k_{\mathrm e_1 + \mathrm e_2}(y)
		\) and
		\(
			k_{\mathrm e_1 + \mathrm e_2}(x)
			=
			k_{\mathrm e_1 - \mathrm e_2}(x)
			=
			x\, k_{2 \mathrm e_2}(1)
		\). So
		\(
			k_{2 \mathrm e_2}(y)
			=
			k_{2 \mathrm e_2}(1)\, y
			=
			y\, k_{2 \mathrm e_2}(y)
		\), and it easily follows that
		\( k_{2 \mathrm e_2}(1) \in K \).
		
		\item
		In the last case
		\( \up 1 {\mathsf D_{2 d, 2}^{(d)}} \) we have
		\[
			k_{\mathrm e_1 - \mathrm e_2}(x)\, y
			=
			x\, k_{2 \mathrm e_2}(y)
			=
			k_{\mathrm e_1 + \mathrm e_2}(x y)
		\]
		for
		\( x \in \mat(d, K) \) and alternating
		\( y \in \mat(d, K) \) (i.e.
		\( y_{ii} = 0 \) and
		\( y_{ij} = - y_{ji} \)). Since
		\( d \geq 2 \) is even (actually, a power of
		\( 2 \)), there is an invertible alternating matrix
		\( J \in \mat(d, K) \), say, the matrix of the split symplectic form. As in the previous case,
		\(
			k_{2 \mathrm e_2}(y)
			=
			k_{\mathrm e_1 + \mathrm e_2}(y)
		\),
		\(
			k_{\mathrm e_1 - \mathrm e_2}(x)
			=
			x\, k_{2 \mathrm e_2}(J)\, J^{-1}
		\),
		\(
			k_{\mathrm e_1 + \mathrm e_2}(x)
			=
			x J^{-1}\, k_{\mathrm e_1 + \mathrm e_2}(J)
		\). Then
		\(
			k_{2 \mathrm e_2}(y)
			=
			k_{2 \mathrm e_2}(J)\, J^{-1} y
			=
			y J^{-1}\, k_{2 \mathrm e_2}(J)
		\) and
		\( k_{2 \mathrm e_2}(J) \) is alternating, so it easily follows that
		\( k_{2 \mathrm e_2}(J)\, J^{-1} \in K \).
	\end{itemize}

	It remains to define ring operations on
	\( \widetilde K \). The group operation is directly induced from
	\( \mathfrak g_\alpha \). The product of
	\( (k_\alpha)_\alpha \) and
	\( (l_\alpha)_\alpha \) is such a family
	\( (m_\alpha)_\alpha \) that
	\[
		[k_\alpha(x_\alpha), l_\beta(x_\beta)]
		=
		m_{\alpha + \beta}([x_\alpha, x_\beta])
	\]
	for
	\( x_\alpha \in X_\alpha \),
	\( x_\beta \in X_\beta \), and roots
	\( \alpha \),
	\( \beta \) such that
	\( \alpha + \beta \) is also a root. It suffices to impose this condition only for one pair
	\( (\alpha, \beta) \) with long
	\( \alpha \).
\end{proof}

\begin{theorem}
	\label{dioph-prbl}
	Let
	\( K \) be a commutative ring with an enumeration by
	\( \mathbb N \) or its finite subset such that the ring operations are computable. Let also
	\( G \) be an isotropic reductive group scheme over
	\( K \) such that the rank of
	\( \Phi \) is at least
	\( 2 \). Then the Diophantine problems for
	\( K \) and
	\( G(K) \) are equivalent, i.e. they reduce to each other by algorithms. Here the enumeration on
	\( G(K) \) is obtained from some closed embedding
	\( G \subseteq \mathbb A^n_K \), it is independent of embedding up to a computable permutation.
\end{theorem}
\begin{proof}
	Clearly, the group operations on
	\( G(K) \) are also computable. Every Diophantine subset of
	\( G(K)^k \) can be considered as a Diophantine subset of
	\( K^{k n} \) using the e-interpretation of
	\( G(K) \) in
	\( K \) and this transformation is clearly computable in terms of defining formulae. In the other direction apply theorem
	\ref{e-interpret}.
\end{proof}

\bibliographystyle{plain}
\bibliography{references}

\end{document}